\documentclass[12pt]{amsart}
\usepackage[T1]{fontenc}
\usepackage{tgtermes}
\usepackage[linkcolor=blue, citecolor = blue, linktocpage = true]{hyperref}
\hypersetup{colorlinks=true}

\usepackage{amsthm,amsfonts,amsmath,amssymb,amscd} 

\usepackage{verbatim}
\usepackage{float}
\usepackage{wrapfig}
\usepackage{mathtools}
\usepackage[color=orange!20, linecolor=orange]{todonotes}

\usepackage{enumitem}

\usepackage{ytableau}

\usepackage{tikz}
\usepackage{tikz-cd} 
\usetikzlibrary{decorations.markings}
\tikzstyle{vertex}=[circle, draw, inner sep=0pt, minimum size=4pt]

\usetikzlibrary{arrows,automata}
\usepackage{tikz, tikz-3dplot}
\usetikzlibrary[positioning,patterns]
\usepackage{caption}
\usepackage{subcaption}
\usepackage{lscape}
\usepackage{rotating}

\newtheorem{theorem}{Theorem}[section]
\newtheorem{proposition}[theorem]{Proposition}
\newtheorem{lemma}[theorem]{Lemma}

\theoremstyle{definition}

\newtheorem{example}[theorem]{Example}

\newtheorem{remark}[theorem]{Remark}

\newcommand{\sgn}{\mathrm{sgn}}
\newcommand{\dett}{\mathrm{det}}

\usepackage{mathabx,epsfig}

\title[Group and algebra hyperdeterminant]
{Group and algebra hyperdeterminant}

\author[Alimzhan Amanov]{Alimzhan Amanov}

\address{
  Kazakh-British Technical University, Kazakhstan
  }
\email{\href{mailto:alimzhan.amanov@gmail.com}{alimzhan.amanov@gmail.com}}

\date{\today}

\begin{document}

\begin{abstract}
    In 1896, Dedekind posed the problem of factoring the group determinant in the non-abelian case to Frobenius, whose solution sparked the birth of finite-group representation theory. Several decades earlier, Cayley introduced the notion of the combinatorial hyperdeterminant of a $d$-way tensor, which is the most natural generalization of an ordinary determinant. In this note, we solve the problem of factoring the group hyperdeterminant. We reduce the computation of the group hyperdeterminant to the computation of the hyperdeterminant at the matrix multiplication tensor and derive a nice closed formula. Further, we extend this notion to associative algebra tensors and show that this polynomial is nonzero if and only if the algebra is semisimple.
\end{abstract}
\maketitle

\section{Introduction}

Let $G = \{g_1,\ldots, g_n\}$ be a finite group of order $|G| = n$. The group determinant is a polynomial in commuting variables $\{x_g \mid g\ \in G\}$ of degree $n$, defined as the determinant of the $n \times n$ matrix $C^{G}_{ij}(x) = x_{g_i g_j}$. This matrix represents the multiplication table of the group algebra of $G$, with its elements replaced by commuting variables.
The group determinant was introduced by Dedekind. He proved that it factorizes into linear forms in the abelian case, which is not true for a non-abelian group. He proposed this problem to Frobenius in 1896, who solved it and subsequently discovered character theory (see \cite{conrad} for a review).
\begin{theorem}[Frobenius \cite{frob}]\label{th:frob}
For any finite group $G$, the group determinant factors as follows:
$$
    \dett_2(C^{G}(x)) = \pm\prod_{\rho \in \widehat{G}} \dett_2\left( \sum_g \rho(g) x_g \right)^{n_\rho},
$$
where $\widehat{G}$ is a complete set of inequivalent irreducible unitary representations, and each polynomial factor is irreducible of total degree $n_{\rho}:=\dim \rho$ and occurs with multiplicity $n_{\rho}$.
\end{theorem}
Total degree of a polynomial is $n = 
\sum_{\rho} n_\rho^2$. 
Moreover, the coefficient of $(x_{\mathrm{id}})^{n_\rho - 1} x_g$ in this irreducible factor 
is equal to $\mathrm{tr}(\rho(g))$, the value of the irreducible character. Other coefficients are known as $k$-characters. Note that the group determinant determines the group \cite{formsibl}.

A tensor $X$ is an element of the vector space $(\mathbb{C}^n)^{\otimes d}$. With the fixed basis in each mode, $X$ can be identified with the hypermatrix $(X_{i_1,\ldots,i_d})_{i_k\in[n]}$, whose entries are indexed by $(i_1,\ldots,i_d)$. Throughout the paper $d$ is assumed to be even.

The \textit{combinatorial hyperdeterminant} was introduced by Cayley \cite{cay,cay2} in 1845 (also referred to as Cayley's first hyperdeterminant) as a natural generalization of the determinant of a matrix to $d$-dimensional hypermatrices. If $X \in (\mathbb{C}^n)^{\otimes d}$ is a $d$-way tensor (hypermatrix), then
\begin{align}\label{eq:hdet}
    \dett_d(X) := \frac{1}{n!}\sum_{\sigma_1,\ldots,\sigma_d \in S_n} \sgn(\sigma_1)\cdots\sgn(\sigma_d) \prod_{i=1}^n X_{\sigma_1(i),\ldots,\sigma_d(i)}.
\end{align}
For odd $d$, the above sum trivially vanishes, while for even $d$ it 
is the unique $\mathrm{SL}(n,\mathbb{C})^{\times d}$-invariant of smallest degree $n$ of the polynomial ring over tensors $(\mathbb{C}^n)^{\otimes d}$, see Section \ref{sec:pre}. 

We study the problem of computing this hyperdeterminant on a natural generalization of a group (and consequently an algebra) hypermatrix.
Define the \textit{group tensor or hypermatrix} $C^G(x)$ by
$$
    (C^G(x))_{g_1,\ldots,g_d} := x_{g_1 \cdot \ldots \cdot g_d},
$$
viewed as an element of $(\mathbb{C}G)^{\otimes d}$ with coefficients in formal variables $\{x_g \mid g \in G \}$.
Similar to the case $d = 2$, this tensor is constructed from the $d$-way multiplication table in the group algebra $\mathbb{C}G$ by replacing $h$ with $x_h$. This tensor carries full information about the group algebra: the slice $(C^G(1,0,\ldots))_{a,b,c,1,1\ldots}$ describes the multiplication tensor of the algebra (solutions of $abc=1$ imply solutions of $ab = c$). It turns out that it admits the following decomposition.
\begin{theorem}[Main theorem]\label{th:main}
    For an arbitrary finite group $G$ and even $d$, we have
    $$
        \dett_d(C^G(x)) =  H_G \cdot \prod_{\rho \in \widehat{G}} \dett_2(\rho(x))^{n_\rho}
    $$
    where $\rho(x) = \sum_g x_g \rho(g)$, $n_\rho = \dim \rho$ and
    $$
        H_G = (-1)^{\frac{d}{2}(|G| - t)/2}  \left(
        |G|^{|G|}\     
        \prod_{\rho \in \widehat{G}} \frac{h_{n_\rho \times n_\rho}}{n_\rho^{n_\rho}}
        \right)^{\frac{d}{2}-1},
    $$
    where $t = \#\{g \in G: g = g^{-1}\}$ and $h_{k \times k} = \prod_{i=0}^{n-1}\frac{(k + i)!}{i!}$. In particular, for $d = 2$, $H_G = \pm 1$.
\end{theorem}
The result is obtained using standard representation theory, properties of the hyperdeterminant, and methods from algebraic combinatorics.
The combinatorial hyperdeterminant has a number of nice properties: formulas, relation to different notions of tensor ranks \cite{ay}, enumeration results related to Selberg integrals \cite{luq-thi}, symmetric functions \cite{matsumoto}, Hamiltonian cycles and intersection of matroids \cite{barvinok}. On the other hand the decision problem of hyperdeterminant vanishing is NP-hard \cite{barvinok} and in general not much is known about its vanishing set. 
Thus, the result above is a numerical representation-theoretic miracle. We illustrate it with a small example.

\begin{example}[Circulant hyperdeterminant]
    Let $G = \mathbb{Z}/n\mathbb{Z} = \mathbb{Z}_n$. Then for $d = 2$ the group matrix $(x_{(i+j)\mathrm{mod}\ n})$ is known as a \textit{circulant matrix}. Its determinant factors into $n$ linear forms over $\mathbb{C}$, each corresponding to one of the $n$ roots of unity $\omega^i=e^{2\pi i/n}$. The following formula is a beautiful generalization of this phenomenon for arbitrary even $d$:
    $$
        \dett_d(x_{(i_1 + \ldots + i_d)\mathrm{mod}\ n}) = \pm\ n^{n(\frac{d}{2} - 1)} \prod_{i=0}^{n-1} \left(\sum_{j=0}^{n-1} x_j\omega^{ij}\right).
    $$
    Since every irreducible representation is $1$-dimensional, each $n_\rho = 1$. This result is itself new. 
    It would be interesting to find an algorithmic way of deriving this result.
\end{example}

The main theorem is proved in two steps. The first is to extract the polynomial part, and the second is the computation of the constant $H_G$, which could potentially vanish. The next subsection establishes its non-vanishing and derives a nice closed formula for $H_G$ by explicit computation of the hyperdeterminant at matrix multiplication tensor.

\subsection{Matrix algebra}

Let the \textit{shifted ($d$-)iterated matrix multiplication tensor} $\mathsf{M}^{(d)}_{X} \in (\mathbb{C}^{n^2})^{\otimes d}$, where $X = (x_{ij}) \in \mathrm{End}(\mathbb{C}^{n})\cong \mathbb{C}^{n^2}$ is the variable matrix, be the tensor
$$    
(\mathsf{M}^{(d)}_{X})_{(i_1,j_1),\ldots,(i_d,j_d)} = \sum \delta_{j_1=i_2}\,\delta_{j_2=i_3}\,\cdots\,\delta_{j_{d-1}=i_d} X_{i_1 j_d} 
 E_{(i_1,j_1)} \otimes \cdots \otimes E_{(i_d,j_d)},
$$
where summation runs over all $d$-tuples of pairs $(i_k,j_k) \in [n]^2$ and $\delta_{x=y}$ is $1$ if $x = y$ and $0$ otherwise. Here, $E_{(i,j)}$ are standard matrix units. 
Alternatively, this tensor has the nice dual form
    $$
        (\mathsf{M}^{(d)}_X)^*(A_1,\ldots,A_d) = \mathrm{Tr}(X^t A_1\cdots A_d).
    $$
Also set $\mathsf{M}^{(d)}_{n}:=\mathsf{M}^{(d)}_{\mathrm{ID}_n}$ and note that $\mathsf{M}^{(d)}_{n}$ is the multiplication tensor of $(d-1)$ matrices $\mu^{(d-1)}_{\mathrm{Mat}_{n\times n}}$. 

The matrix multiplication tensor has been extensively studied from the point of view of (geometric) complexity theory, which seeks faster algorithms for matrix multiplication \cite{blas, land}. In particular, algorithms with fewer than $O(n^3)$ multiplications (but not the best known) were obtained in the framework of the Cohn--Umans program by studying the group algebra tensor $\mu^{(2)}_G$ \cite{cohn}. 

To prove Theorem~\ref{th:main}, we show in fact that $C^G(x) \cong \bigoplus \mathsf{M}^{(d)}_{X_{\rho}}$, where $X_{\rho} = \sum \rho(g) x_g$; see Theorem~\ref{th:block}.
The constants $H_G$ are in fact certain differences of `positive' and `negative' objects, and it is frequently a mystery why the difference is not zero.
This issue is resolved by explicitly computing the hyperdeterminant at the tensor $\mathsf{M}_n^{(d)}$. 
\begin{theorem}\label{th:mmt}
    For even $d$ and any $n$:
    $$
        \dett_d(\mathsf{M}^{(d)}_{n}) =  \pm\left(
            \prod_{i=0}^{n-1} \frac{(n+i)!}{i!}
            \right)^{\frac{d}{2} - 1}.
    $$
\end{theorem}
The result is obtained using powers of the \textit{Young symmetrizers} $c_\lambda$, the idempotents of the group algebra $\mathbb{C}S_n$, which is somewhat surprising since the initial problem does not appear to involve the symmetric group. Also, the theorem is the first instance of nonvanishing of $\mathrm{M}_n^{(d)}$ (or $\mathrm{M}_X^{(d)}$) for a hyperdeterminant, albeit only for even $d$. The value is exactly $h_{n \times n}$ -- the product of hooks in a square.

\vspace{0.5em}
We will in fact prove a more general result.

\subsection{Semisimple algebras}
Let $A$ be a finite-dimensional associative semisimple algebra with a fixed basis $\{a_1,\ldots,a_n\}$ and product rule $a_{i_1} \cdot \ldots \cdot a_{i_d}  = \sum \mu^k_{i_1 \ldots i_d} a_k$, where\footnote{Canonically $\mu \in (A^{*})^{\otimes d} \otimes A$, equivalently $\mu \in (\mathbb{C}^n)^d\otimes (\mathbb{C}^n)^*$ as long as the basis and $A^*$ are fixed.} $\mu^k_{i_1,\ldots,i_d} \in (\mathbb{C}^n)^{\otimes (d+1)}$. Define the \textit{algebra tensor (hypermatrix)}
to be the following $n^{\times d}$ hypermatrix $$
    C^A_{i_1,\ldots,i_d}(x) := \sum \mu^k_{i_1,\ldots,i_d} x_k.
$$ 
Alternatively $C^A = (\ell_x(a_{i_1}\cdots a_{i_d}))_{i_k \in [n]}$ for a generic linear form $\ell_x(a) = \sum x_i a^*_i: A \to \mathbb{C}$. 

Let $\widehat{A}$ be the set of representatives $\rho: A \to \mathrm{End}(V_\rho)$ of isomorphism classes of irreducible representations (or simple modules), see \cite{etingof}. Set $n_\rho = \dim V_\rho$ and let $E_{ij}^{\rho} \in \mathrm{End}(V_\rho) \cong \mathbb{C}^{n_\rho^2}$ be the matrix-units basis. Fix a Wedderburn algebra isomorphism
    $\Phi:=\bigoplus_{\rho \in \widehat{A}} \rho: A \to \bigoplus_{\rho \in \widehat{A}} \mathrm{End}(V_\rho)$.

\begin{theorem}\label{th:alg}
Let $A$ be semisimple algebra with the given basis $\{a_1,\ldots, a_n \}$. Then,
    $$
        \dett_d(C^A(x)) 
        = H_{A} \cdot \prod_{\rho \in \widehat{A}} \left(\ \det(y_{ij}^{\rho})\ \right)^{n_\rho},
    $$
    where $y_{ij}^{\rho} := \ell_x(U x)$ 
    and $H_A = \det(\Phi)^d (n^n \prod_{\rho \in \widehat{A}}h_{n_\rho \times n_\rho}/n_\rho^{n_\rho})^{d/2-1}$. In particular, each polynomial factor does not depend on the choice of matrix units in $\mathrm{End}(V_\rho)$.
\end{theorem}

This result refines the previous theorems (up to a constant): if $A = \mathbb{C}S_n$ and $\{a_k\} = S_n$, then $C^A(x) = C^G(x)$, and coordinates $y_{ij}^{\rho}$ represent scaled Fourier coordinates.
If $A = \mathrm{Mat}(n, \mathbb{C})$ and $\{a_k\} = \{E_{ij}\}$, then $C^A(x) = \mathsf{M}_X^{(d)}$. The downside of this statement is that we do not have a canonical basis of $A$, so we have to bring a concrete $\Phi$ into the picture. Also, the constant $H_A$ is determined only up to a factor: it depends on the chosen basis (through the constant $\beta$), but still $H_A \neq 0$. Note that if $\Phi$ is an automorphism (i.e. the $a$-basis is conjugate-equivalent to the matrix units by the Skolem--Noether theorem), then $\det(\Phi) = \pm 1$. It is the only concrete invariant known to be nonvanishing on a semisimple tensor, yet only for even $d$.

\subsection{Nonsemisimple algebras}
The algebra hyperdeterminant is a relative invariant with respect to the group $\mathrm{GL}(A)$, so its vanishing does not depend on the basis of $A$ that we choose; see Lemma~\ref{lemma:basis}. Frobenius' criterion states that $\det_2(C^A(x)) \neq 0$ (as a polynomial) if and only if the algebra is Frobenius, i.e. for some $x$ the associated bilinear form $ab \mapsto \ell_x(ab)$ is nondegenerate; alternatively, there exists an associative scalar product.
It is referred to as \textit{parastrophic determinant} of $A$ in the works of Brauer, Nakayama and Nesbitt, see also \cite{stein} (and references therein) for alternative extensions of the notion. 

We show that the algebra hyperdeterminant vanishes if the algebra is not semisimple, i.e. if its Jacobson radical is nonzero. This shows that the semi-invariant is more restrictive for $d > 2$.
\begin{theorem}
    Let $A$ be an associative finite dimensional algebra. For $d > 2$, 
    $\det_d(C^A(x)) \neq 0$ if and only if the algebra $A$ is semisimple.
\end{theorem}
The result is obtained by geometric methods; in particular, we show that the $\mathrm{SL}^{\times d}$-orbit closure of $C^A(x)$ contains $0$, which is a stronger statement.
This theorem completes the characterization of the vanishing of the combinatorial hyperdeterminant on associative algebra tensors $C^A(x)$. Similar results hold for iterated multiplication tensors $\mu_A$.
The vanishing of the combinatorial hyperdeterminant is an important open problem with interesting consequences; see \cite{rota, kum, yel, zappa}.

It is known that non-semisimple multiplication tensors are unstable \cite{lys}, i.e. their orbit closure contains the zero tensor (equivalently, they vanish on every hyperdeterminant). In turn, semisimple tensors are semistable (moreover polystable), i.e. they do not vanish on some hyperdeterminant, as follows from the methods of geometric invariant theory. From this point of view, the tensor $C^A(x)$ exhibits similar behavior, although it is slightly more general.

Nevertheless, concrete nonvanishing certificates were not known. For instance, another famous invariant, the \textit{geometric hyperdeterminant} (or GKZ hyperdeterminant), discovered by Gelfand, Kapranov, and Zelevinsky \cite{gkz}, vanishes whenever the tensor is degenerate, which is the case (except for few small cases) for semisimple algebra multiplication tensors (and tensors $C^A(x)$).

The paper is structured as follows. In Section~\ref{sec:pre}, we recall standard facts and properties of the combinatorial hyperdeterminant. In Section~\ref{sec:block}, we prove the block-diagonalization part of the main theorem for the group tensor. In Section~\ref{sec:mat}, we compute the matrix algebra hyperdeterminant (Theorem~\ref{th:mmt}) and complete the proof of the main theorem. In Section~\ref{sec:algebra}, we consider algebra tensors in the semisimple and non-semisimple cases.

\section{Preliminaries}\label{sec:pre}
We denote $[n] = \{1,2,\ldots, n\}$. 

\subsection{Tensors and hypermatrices}\label{sec:tensors} An element of the space $\mathbb{C}^{n} \otimes \cdots \otimes \mathbb{C}^{n} = (\mathbb{C}^{n})^{\otimes d}$ is called a tensor.
The group $G = \mathrm{SL}(n,\mathbb{C}) \times \cdots \times \mathrm{SL}(n,\mathbb{C}) = \mathrm{SL}(n,\mathbb{C})^{\times d}$ naturally acts on the space of tensors $(\mathbb{C}^{n})^{\otimes d}$ by $(g_1,\ldots, g_d) \cdot v_1 \otimes \cdots \otimes v_d = g_1 v_1 \otimes \ldots \otimes g_d v_d$ for $g_i \in \mathrm{SL}(n_i)$ on simple tensors, and the action is extended multilinearly.
The elements of $(\mathbb{C}^{n})^{\otimes d}$ written in a fixed basis in each mode correspond to hypermatrices $(X_{i_1,\ldots, i_d})$ indexed by $(i_1,\ldots, i_d) \in [n]^d$, and we shall usually identify tensors in $(\mathbb{C}^{n})^{\otimes d}$ with the corresponding hypermatrices.

\subsection{Hyperdeterminants}
A. Cayley regarded all $\mathrm{SL}(n,\mathbb{C})^{\times d}$-invariant (homogeneous) polynomials over tensors as \textit{hyperdeterminants}. 
The combinatorial hyperdeterminant, defined in \eqref{eq:hdet}, has a nice combinatorial formula and also inherits several properties of the ordinary determinant that are important for us and that we list below.

\begin{proposition}[SL-invariance]\label{prop:inv}
    For even $d$ and arbitrary $n$, we have $$\dett_d((g_1,\ldots,g_d) \cdot X) = \det(g_1)\cdots\det(g_d)\dett_d(X)$$ for any $X \in (\mathbb{C}^n)^{\otimes d}$ and $g = (g_1, \ldots, g_d) \in \mathrm{GL}(n,\mathbb{C})^{\times d}$.
\end{proposition}

Let $Y \in (\mathbb{C}^{a})^{\otimes d}$ and $Z \in (\mathbb{C}^{b})^{\otimes d}$ with $a + b = n$. The \textit{direct sum of tensors} $Y \oplus Z \in (\mathbb{C}^n)^{\otimes d}$ is the tensor that results from placing tensors block diagonally, i.e. $Y$ occupies the subcube $[1,a]^d$ and $Z$ occupies the subcube $[a+1,a+b]^d$ of the full cube $[n]^d$. 

\begin{proposition}\label{prop:direct}
    For any $Y \in (\mathbb{C}^{n_1})^{\otimes d}$ and $Z \in (\mathbb{C}^{n_2})^{\otimes d}$, we have
    $$
        \dett_d(Y \oplus Z) = \dett_d(Y) \cdot \dett_d(Z).
    $$
\end{proposition}
\begin{proof}
    Set $X = Y \oplus Z$.
    Let us pack the tuple of permutations $(\sigma_1,\ldots,\sigma_d) \in S_n^d$ figuring in the definition of $\dett_d(X)$ as a map $\sigma: [n] \to [n]^d$, i.e. rook placements in hypercube, where each rook hits the entire slice. 
    Expanding the formula without the factor $(n!)^{-1}$ by forcing $\sigma_1 = \mathrm{id}$
    $$
        \dett_d(X) = \sum_{\sigma: [n+m] \to [n+m]^d, \sigma_1 = \mathrm{id}} \sgn(\sigma) \prod_{i=1}^{n+m}X_{\sigma(i)},
    $$
    where $\sgn(\sigma) = \prod_{\ell} \sgn(\sigma_\ell)$. Note that the term $X_{\sigma(i)} = 0$ vanishes whenever $\sigma \not\in [1,n]^d \cup [n+1,m]^d$. Thus, any $\sigma$ bijectively splits into the pair $\sigma_y:= \sigma([1,n])$ and $\sigma_z:=\sigma([n+1,n+m]) - n$, where $\sigma_y: [n] \to [n]^d$ and $\sigma_z: [m] \to [m]^d$, with $\sgn(\sigma) = \sgn(\sigma_y) \sgn(\sigma_z)$. Hence,
    $$
        \dett_d(Y \oplus Z) = \sum_{\sigma_y:[n]\to[n]^d} \sum_{\sigma_z:[m]\to[m]^d} \sgn(\sigma_y)\sgn(\sigma_z) \prod_{i=1}^n Y_{\sigma_y(i)} \prod_{i=1}^m Z_{\sigma_z(i)} = \dett_d(Y) \cdot \dett_d(Z).
    $$
    This completes the proof.
\end{proof}
There are other nice properties, such as the Cauchy--Binet formula and Laplace expansion by minors, which we will not use in this paper. On the other hand, other hyperdeterminants are usually not multiplicative with respect to the direct-sum operation. From this point of view, the invariant $\dett_d$ is special, since it is the unique one that is linear in the slices, just as $\dett_2$ is linear in rows and columns. It exists only when $d$ is even. When $d$ is odd, this function vanishes trivially, and the minimal degree of an $SL$-invariant polynomial on $(\mathbb{C}^n)^{\otimes d}$ is at least $n\lceil \sqrt{n} \rceil$; see \cite{ay2, bi}. In particular, for $d=3$ and $n = k^2$, the invariant of minimal degree $k^3$ is also unique, which is a `coincidence': the invariant lives on $(\mathbb{C}^{k^2})^{\otimes 3}$, the space where the matrix multiplication tensor lives. Its evaluation on the matrix multiplication tensor is considered in \cite{bi, lzx}, but its vanishing behavior is not known.

\section{Tensor block-diagonalization}\label{sec:block}

In this section, we prove the block diagonalization of the group algebra tensor,
which reflects the fact that the group algebra decomposes into a direct sum of matrix algebras by the Artin--Wedderburn theorem. In particular, we will show that the group algebra tensor $C^G$ decomposes into a direct sum of matrix algebra tensors.
\begin{theorem}[Tensor block-diagonalization]\label{th:block}
    Let $G$ be a finite group of order $n$ and let $d \ge 2$. There is a unitary transformation $U \in \mathrm{U}(n, \mathbb{C})$ such that
    \begin{align}\label{eq:direct}
        (U, \ldots, U) \cdot C^G(x) = \left( \bigoplus_{\rho \in \hat{G}} \gamma_\rho \mathsf{M}^{(d)}_{\rho(x)}\right)
    \end{align}
    for explicit constants $\gamma_\rho \in \mathbb{C}^{\times}$, where $\rho(x) = \sum_g x_g \rho(g) \in \mathrm{End}(\mathbb{C}^{n_\rho})$.
\end{theorem}
\begin{proof}
    By the Artin-Wedderburn decomposition, we have:
    $$
        \mathbb{C}G \cong \bigoplus_{\rho \in \hat{G}} \mathrm{End}(\mathbb{C}^{n_{\rho}}), \qquad \mathrm{End}(\mathbb{C}^{n_{\rho}}) = \langle E^{\rho}_{ij} \mid i,j \in [n_\rho]\rangle
    $$
    thus, the desired space $\mathbb{C}G^{\otimes d}$, in which $C$ lives, decomposes as follows
    $$
        C^G \in \mathbb{C}G^{\otimes d} \cong \bigoplus_{\rho_1,\ldots,\rho_d \in \hat{G}} \mathrm{End}(\mathbb{C}^{n_{\rho_1}}) \otimes \ldots \otimes \mathrm{End}(\mathbb{C}^{n_{\rho_d}}),
    $$
    i.e. it decomposes into a direct sum of tensor products of matrix algebras. We regard $x_i \in \mathbb{C}$ here. To turn $C^G$ into this form, we apply the so-called \textit{non-commutative group Fourier transform} to each mode of the tensor $C^G$. Let $U$ be the $|G| \times |G|$ matrix with rows indexed by triplets $(\rho, i, j) \in \hat{G} \times [n_{\rho}]^2$ and columns by $g \in G$:
    \begin{align*}
        U_{(\rho, i, j), g} := \alpha_{\rho}^{-1/2} \cdot \rho(g)_{ij}, \quad \text{where }\alpha_\rho = \frac{|G|}{n_{\rho}}.
    \end{align*}
    Recall the Schur orthogonality in a matrix form. For any triplets $(\rho, i, j)$ and $(\sigma, k, \ell)$:
    \begin{align}\label{eq:orth}
        \sum_{g \in G} \rho(g)_{ij} \widebar{\sigma(g)_{k\ell}} = \delta_{\rho \sigma} \delta_{ik} \delta_{j \ell} \cdot \alpha_\rho,
    \end{align}
    Hence, $U U^* = I$ and $U$ is unitary. In particular, $|\det(U)| = 1$.
    Now, act by $U$ in each tensor component, set the result to be $B := (U, \ldots, U) \cdot C^G$. 
    Let $C^G_h(x) = \sum_{g_1\cdots g_d = h} x_{h} g_1 \otimes\ldots\otimes g_d$, so that $C^{G}(x) = \sum_{h \in G} x_h C^G_h(x)$. Then, $C^G_h(x)$ transforms to
$$
    U^{\otimes d} \cdot C^G_h(x) 
    = \prod_{i=1}^d \alpha^{-1/2}_{\rho_i}\sum_{g_1\cdots g_d = h} x_h \sum_{(\rho_k,i_k,j_k)} \prod_{k=1}^d\rho_k(g_k)_{i_k j_k} \bigotimes_{k=1}^d E^{\rho_k}_{i_k j_k}
$$
    where the sum runs over triplets $(\rho_k,i_k,j_k)$ for each $k = 1, \ldots, d$. Set $\beta := \prod_{i=1}^d \alpha^{-1/2}_{\rho_i}$. Thus, the coordinates of $B$ are
    \begin{align*}
        B_{(\rho_1, i_1, j_1), \ldots, (\rho_d, i_d, j_d)} =
         \sum_{h \in G} x_h \beta \sum_{g_1, \ldots, g_{d-1} \in G} \rho_d( (g_1\cdots g_{d-1})^{-1} h)_{i_d j_d} \prod_{k=1}^{d-1} \rho_k(g_k)_{i_k j_k}.
    \end{align*}
    Denote by $c_h$ the coefficient of $x_h$ in the latter sum. 
    Since $\rho$ is a homomorphism and $\rho(g^{-1})_{ij} = \widebar{\rho(g)}_{ji}$, we expand 
    $$
        \rho_d( (g_1\cdots g_{d-1})^{-1} h)_{i_d j_d} = \sum_{\ell_1,\ldots,\ell_{d-1},\ell_d=i_d} \prod_{k=1}^{d-1}\widebar{\rho_d(g_k)}_{\ell_k,\ell_{k+1}} \cdot \rho_d(h)_{\ell_1, j_d},
    $$
    and proceed to compute the coefficient $c_h$
    \begin{align*}
        c_h &=
        \beta  \sum_{\ell_1,\ldots,\ell_{d}=i_d} 
        \rho_d(h)_{\ell_1 j_d} 
        \sum_{g_1, \ldots, g_{d-1}}  
            \prod_{k=1}^{d-1}
            \widebar{\rho_d(g_k)}_{\ell_k \ell_{k+1}} \rho_k(g_k)_{i_k j_k}
        \\
        &= \beta  \sum_{\ell_1,\ldots,\ell_{d}=i_d} 
        \rho_d(h)_{\ell_1 j_d} 
        \prod_{k=1}^{d-1}
        \left(\sum_{g}
            \widebar{\rho_d(g)}_{\ell_k \ell_{k+1}}
            \rho_k(g)_{i_k j_k}
        \right)
        \\
        &= \beta \alpha^{d-1}_{\rho_d} \sum_{\ell_1,\ldots,\ell_{d}=i_d} 
        \rho_d(h)_{\ell_1 j_d} 
        \prod_{k=1}^{d-1}
            \delta_{\rho_d, \rho_k}
            \delta_{\ell_k, i_{k}}
            \delta_{\ell_{k+1}, j_k}.
    \end{align*}
    By the Schur orthogonality relations \eqref{eq:orth}, the latter term in the sum vanishes unless $\ell_1 = i_1$, $j_1 = i_2$, $j_2 = i_3$, $\ldots$, $j_{d-1} = i_d$ and $\rho_k = \rho$, which reduces the sum to a single term
    $$
        c_h = \alpha_\rho^{d/2-1} \rho(h)_{i_1, j_d} \prod_{k=1}^{d-1} \delta_{j_k, i_{k+1}}
    $$
    for any $h \in G$. Thus, $B_{(\rho_1,i_1,j_1),\ldots,(\rho_d,i_d,j_d)}$ vanishes unless $\rho_k = \rho_1$, which means that $B$ is block-diagonal. There are $|\hat{G}|$ blocks and each block represents $n^2_{\rho} \times \ldots \times n^2_{\rho}$ hypermatrix, which we denote by $B^{\rho}$. 
    If we set the matrix $\rho(x) := \sum_{h\in G} x_h \rho(h)$, then the coordinates are
    $$
        B^{\rho}_{(i_1,j_1),\ldots,(i_d,j_d)} 
        = \alpha_\rho^{d/2-1}\rho(x)_{i_1, j_d} \prod_{k=1}^{d-1} \delta_{j_k, i_{k+1}} 
        = \alpha_\rho^{d/2-1} \mathsf{M}^{(d)}_{\rho(x)}.
    $$
    In particular, $B = \bigoplus_{\rho} B^{\rho} =  \bigoplus_{\rho} \alpha_\rho^{d/2-1}\mathsf{M}^{(d)}_{\rho(x)}$, and $\gamma_\rho = \alpha_{\rho}^{d/2-1}$ as desired.
\end{proof}
\vspace{1em}
\begin{remark}
    For $V := \mathrm{End}(\mathbb{C}^{n_\rho})$, as a multilinear form in $V^{\otimes d}$, a single block is equal to
$$(B^\rho)^*(A^1,\ldots,A^d) = \sum_{i_1,\ldots,i_d,j_d} \rho(x)_{i_1,j_d} A^1_{i_1, i_2}\cdot A^2_{i_2, i_3} \cdots A^d_{i_{d}, j_d} = \mathrm{tr}(\rho(x)^t A^1 \cdots A^d),
$$
i.e. dual form of matrix multiplication tensor.
\end{remark}
\begin{remark}
    The theorem above together with Lemma~\ref{lemma:trace} implies that for any homogeneous hyperdeterminant $\Delta \in \mathbb{C}[(\mathbb{C}^n)^{\otimes d}]^{SL}_{nk}$, we have
    $$
        \Delta(C^G(x)) = \left(\gamma^n \prod_{\rho \in \widehat{G}}\dett_2(\rho(x))^{n_\rho}\right)^{k} \Delta\left(\scalebox{0.9}{$\bigoplus$}\,\mathsf{M}_{n_\rho}\right).
    $$
    Yet even for a single $\mathsf{M}_n$, the computation of $\Delta(\mathsf{M}_n)$ is a nontrivial problem.
\end{remark}

\begin{remark}
    Note that Theorem~\ref{th:block} is true for odd $d$ as well, but in this paper we are interested only in the case of even $d$. The reason is that the invariant ring of odd $d$-way tensors has no linear-degree invariants.
\end{remark}

In the following lemma, we treat the constant $\det(U)^d$ appearing after applying $U$ to $C^G$.
\begin{lemma}\label{lemma:sign}
    For a $|G| \times |G|$ matrix of group Fourier transform $U_{(\rho,i,j), g} = \left(\frac{n_{\rho}}{|G|}\right)^{1/2} \rho(g)_{ij}$, we have
    $$
        \dett_2(U)^2 = (-1)^{(|G| - t)/2 + \sum \binom{n_\rho}{2}}, 
    $$
    where $t = \#\{g \in G: g^{-1} = g\}$ is the number of involutions.
\end{lemma}
\begin{proof}
    Since $\rho$ is unitary, we have $\widebar{\rho(g)}_{ij} = \rho(g^{-1})_{ji}$; therefore
    $$
        \widebar{U}_{(\rho,i,j), g} = U_{(\rho,j,i), g^{-1}}.
    $$
    Hence, $\widebar{U}$ is obtained from $U$ by permuting rows and columns. Explicitly, if $C$ is the permutation matrix acting on columns induced by $g \mapsto g^{-1}$ and $R$ is the permutation matrix acting on rows induced by $(\rho,i,j) \mapsto (\rho,j,i)$, then $\widebar{U} = R U C$. Using $U U^* = 1$, we conclude that $\det_2(U)^2 = \det(R)\det(C) = (-1)^{r} (-1)^c$, where $r$ and $c$ are the numbers of swaps performed by $R$ and $C$. It is easy to see that $\det(R) = (-1)^{(|G| - t)/2}$ and $\det(C) = (-1)^{\sum (n^2_{\rho} - n_{\rho})/2}$.
\end{proof}
\begin{remark}
    For symmetric group the quantity $\sum_{\rho} n_\rho = t$ is the number of involutions.
\end{remark}

\section{Matrix algebra tensor}\label{sec:mat}
Fix even $d$ and let us write $\mathsf{M}_X := \mathsf{M}^{(d)}_X$.
\begin{lemma}\label{lemma:trace}
    Let $X = (x_{ij})$ be a $n \times n$ matrix. The tensor $\mathsf{M}_X\in (\mathbb{C}^{n^2})^{\otimes d}$ satisfies $$\dett_d(\mathsf{M}_X) = \det(X)^n \dett_d(\mathsf{M}_n).$$
\end{lemma}
\begin{proof}
It is not hard to see that $$\mathsf{M}^*_X(A_1,\ldots,A_d) = \mathsf{M}^*_{I_n}(XA_1, \ldots, A_d) = R_X \cdot \mathsf{M}^*_{I_n}(A_1, \ldots, A_d),$$ where $R_X \in \mathrm{End}(\mathbb{C}^{n^2})$ is a linear operator acting on the first mode of tensor $\mathsf{M}^*_{I_n}$ by $A_1 \to X \cdot A_1$. The hyperdeterminant is a relative invariant (Proposition~\ref{prop:inv}), so 
$$
    \dett_d(\mathsf{M}_X) 
    = \dett_d(R_X \cdot \mathsf{M}_{I_n}) 
    = \det(R_X) \cdot \dett_d(\mathsf{M}_{I_n}).
$$
The operator $R_X$ acts by matrix multiplication $X$ on the column subspace spanned by the vectors $\{E_{1,j},\ldots,E_{n,j}\}$ independently for each column $j$, thus  
$\det(R_X) = \det(X \otimes I_{n}) = \det(X)^n$. 
\end{proof}

It remains to compute $\det(\mathsf{M}_{n_{\rho}})$.

\begin{theorem}[cf. Theorem~\ref{th:mmt}]
    For even $d$ and any $n$:
    $$
        \dett_d(\mathsf{M}_{n}) =  \pm\left(
            \prod_{i=0}^{n-1} \frac{(n+i)!}{i!}
            \right)^{\frac{d}{2} - 1}.
    $$
\end{theorem}
\begin{proof}
    To compute $\det_d(\mathsf{M}_{n})$ we expand:
    $$
        \dett_d(\mathsf{M}_{n}) = \sum_{\sigma_1=\mathrm{id},\sigma_2,\ldots,\sigma_d \in S_{[n]^2}} \sgn(\sigma_1\cdots\sigma_d) \prod_{(i,j) \in [n]^2} (\mathsf{M}_n)_{\sigma_1(i, j),\ldots,\sigma_d(i, j)},
    $$
    where the sum runs over permutations in $S_{[n]^2}$. For a given tuple $(\sigma_i)$, the corresponding monomial does not vanish if and only if
    \begin{align}\label{eq:conditions}
        \sigma_i(p)_2 = \sigma_{i+1}(p)_1
    \end{align}
    for any $p \in [n]^2$ and $i \in [d]$ ($\sigma_{d+1} = \sigma_1$), i.e. the second coordinate of $\sigma_i$ and the first coordinate of $\sigma_{i+1}$ are identical. 
    Under the standard lexicographical order, let us present each $\sigma$ as $2$-row array:
    $$
        \sigma := \begin{bmatrix}
            \sigma(1,1)_1 & \sigma(1,2)_1 & \ldots & \sigma(n,n)_1\\
            \sigma(1,1)_2 & \sigma(1,2)_2 & \ldots & \sigma(n,n)_{2}
        \end{bmatrix}.
    $$
    In particular, the identity permutation reads as
    $$
    \mathrm{id} = \begin{bmatrix}
            11\ldots 1 \ldots nn\ldots n\\
            12\ldots n \ldots 12\ldots n
        \end{bmatrix} =: \begin{bmatrix}
            \mathrm{row}\\
            \mathrm{col}
        \end{bmatrix},
    $$
    where the position $p$ is occupied with the vertical pair $(i,j)$ so that $(i-1)n + j = p$, and the words $\mathrm{row}$ and $\mathrm{col}$ represent the first and the second rows of $\mathrm{id}$.  
    By the conditions imposed in \eqref{eq:conditions} we may write:
    $$
    \sigma_1 = \begin{bmatrix}
             a_1  \\
             a_2  
        \end{bmatrix},  
    \qquad
    \sigma_2 = \begin{bmatrix}
             a_2  \\
             a_3  
        \end{bmatrix},
    \qquad
    \ldots,
    \qquad
    \sigma_d = \begin{bmatrix}
             a_{d}  \\
             a_1  
        \end{bmatrix}
    $$
    where $a_1 = \mathrm{row}$, $a_2 = \mathrm{col}$ and we iterate over the words $a_3,\ldots a_{d} := S_{n^2} \cdot (1^n \ldots n^n)$ --- anagrams of the word $\mathrm{row}$. Let $R$ and $C$ be the row and column stabilizers of the words $\mathrm{row}$ and $\mathrm{col}$ in $S_{n^2}$. 
    Let $\tau \in S_{[n]^2}$ be the transpose permutation, i.e. $\tau(i,j) = \tau(j,i)$. Then
    \begin{align}\label{eq:sigma-exp}
        \sigma_{i-1} = \pi_{i-1}\cdot \tau \cdot \sigma_{i}, \qquad \text{where}\quad \pi_{i-1}: \begin{bmatrix}
            a_{i+1}\\a_{i}
        \end{bmatrix} \to \begin{bmatrix}
            a_{i-1}\\a_{i}
        \end{bmatrix}.
    \end{align}
    Indeed, diagrammatically:
    $$
    \begin{bmatrix}
        \mathrm{row}\\\mathrm{col}
    \end{bmatrix}
    \xlongrightarrow{\sigma_{i}}
    \begin{bmatrix}
        a_{i}\\a_{i+1}
    \end{bmatrix}
    \xlongrightarrow{\tau}
    \begin{bmatrix}
        a_{i+1}\\a_{i}
    \end{bmatrix}
    \xlongrightarrow{\pi_{i-1}}
    \begin{bmatrix}
        a_{i-1}\\a_{i}
    \end{bmatrix}.
    $$
    Moreover, $\pi_{i-1} \in C$, since the set $\{(1,c),\ldots,(n,c)\}$ is permuted independently for each column $c$. Further, by repeated application of \eqref{eq:sigma-exp} we obtain:
    $$
        \sigma_{1} = \pi_1\tau\sigma_{2} = \pi_{1}\tau\pi_{2}\tau\sigma_{3} = \ldots = \pi_{1}\tau\pi_{2}\tau\cdots\pi_d\tau\sigma_{1},
    $$
    which yields
    \begin{align}\label{eq:pi-prod}
        \pi_{1} \cdot \tau \cdot \ldots \cdot \pi_d \cdot \tau 
        = \prod_{i=1}^{d/2} \pi_{2i-1} \cdot (\tau \pi_{2i} \tau)
        = \mathrm{id},
    \end{align}
    where each $\pi_i \in C$. 
    Note that from the tuple $(\pi_1,\ldots,\pi_d)$ satisfying \eqref{eq:pi-prod} we can uniquely recover the tuple $(\sigma_1,\ldots,\sigma_d)$ with $\sigma_1=\mathrm{id}$ satisfying \eqref{eq:sigma-exp}. Therefore, this transition is bijective. By grouping $\sigma_{2i-1}\sigma_{2i}$ and using \eqref{eq:sigma-exp}, we track the sign of the monomial:
    $$
        \prod_{j=1}^d\sgn(\sigma_{j}) 
        = \prod_{i=1}^{d/2} \sgn(\sigma_{2i-1})\sgn(\sigma_{2i})
        = \sgn(\tau)^{d/2} \prod_{i=1}^{d/2} \sgn(\pi_{2i}).
    $$

    For even $i=1,\ldots,d/2$ replace $\pi_{2i}$ with $\pi'_{2i} = \tau \pi_{2i} \tau \in R$ in \eqref{eq:pi-prod}, i.e., with its conjugation by the involution $\tau = \tau^{-1}$. This allows us to rewrite the identity as 
    $$
    \prod_{i=1}^{d/2} \pi_{2i-1}\pi'_{2i} = \mathrm{id}.
    $$
    By $\sgn(\tau) = (-1)^{\binom{n}{2}}
    $ and $\sgn(\pi'_i) = \sgn(\pi_i)$ we transform the initial sum to
    \begin{align}\label{eq:signed}
        \det(\mathsf{M}_{I_n}) 
        = (-1)^{\frac{d}{2}\binom{n}{2}}
        \cdot 
        \sum_{
            \pi_i \in C:\ 
            \pi_1 \pi'_{2} \cdots \pi_{d-1} \pi'_{d} = \mathrm{id}
        } \sgn(\pi_2 \pi_4 \cdots \pi_d).
    \end{align}
    We will show that this sum appears in a power of the Young symmetrizer, an element of $\mathbb{C}S_n$ used to construct irreducible representations of $S_{n^2}$ (see \cite{fulton}). Namely, consider the following elements of the group algebra $\mathbb{C}S_{n^2}$ for $\lambda = n \times n$:
    $$
        a_{\lambda} = \sum_{g \in C} g, \qquad b_{\lambda} = \sum_{g \in R} \sgn(g) g, \qquad c_{\lambda} = a_{\lambda} b_{\lambda}.
    $$
    It is known that $(c_{\lambda})^2 = h_{\lambda} c_{\lambda}$ and its coefficient at the identity is $c_{\lambda}(\mathrm{id}) = 1$, where $h_{\lambda} = \prod_{(i,j) \in \lambda}(i+j-1)$ is the product of hooks. Consider its power:
    $$
        (c_{\lambda})^{d/2} = \prod_{i=1}^{d/2}\left(\sum_{\pi_{2i-1} \in C} \pi_{2i-1}\right) \cdot \left(\sum_{\pi'_{2i} \in R} \sgn(g) \pi'_{2i}\right) = \sum_{\pi_{i} \in C}\sgn(\pi'_{2}\cdots \pi'_{d}).
    $$
    Hence, evaluation at the identity yields
    $$(c_{\lambda})^{d/2}(\mathrm{id}) = (-1)^{\frac{d}{2}\binom{n}{2}}\det(\mathsf{M}_{n}).
    $$ 
    On the other hand, $(c_{\lambda})^{d/2} = h_{\lambda}^{d/2-1} \cdot c_{\lambda}$, thus
    $$
        \det(\mathsf{M}_n) = \pm (h_{n \times n})^{d/2 - 1} \cdot c_{\lambda}(1) = \pm (h_{n \times n})^{d/2 - 1},
    $$
    where $h_{n \times n} = \prod_{i=0}^{n-1} \frac{(n+i)!}{i!}$ is the product of hooks in the square partition.
\end{proof}
\begin{remark}
    More general signed sums such as \eqref{eq:signed} appear frequently in the evaluation of other hyperdeterminants at distinct tensors, see \cite{ay2,  bi, lzx}.
\end{remark}
We now complete the proof of the main theorem.
\begin{proof}[Proof of Theorem~\ref{th:main}]
By Theorem~\ref{th:block}, the Cayley hypertable $C^G$ is unitarily equivalent to a direct sum of matrix multiplication tensors
$
    (U,\ldots, U) \cdot C^G(x) = \bigoplus \alpha_\rho^{d/2-1} \mathsf{M}_{\rho(x)}.
$
The $\mathrm{SL}$-invariance and direct-sum properties of the hyperdeterminant (Proposition~\ref{prop:inv}, Proposition~\ref{prop:direct}) and concrete computation of sign factor (Lemma~\ref{lemma:sign}), imply that
$$
    \dett_d(C^G(x)) 
    = (-1)^{d s/2} \prod_{\rho \in G} \dett_d(\alpha^{d/2-1}_{\rho} \mathsf{M}^{(d)}_{\rho(x)})
    = (-1)^{d s/2} \prod_{\rho \in G} (\alpha^{d/2-1}_{\rho})^{n^2_\rho} \dett_d(\mathsf{M}^{(d)}_{\rho(x)}),
$$
where $s = \sum_{\rho} \binom{n_\rho}{2} + (|G| - t)/2$. 
Recall, that $\alpha_{\rho} = |G| / n_{\rho}$.
By Lemma~\ref{lemma:trace}, we compute each factor $\dett_d(\mathsf{M}^{(d)}_{\rho(x)}) = \det(\rho(x))^{n_\rho} \dett_d(\mathsf{M}^{(d)}_{n_\rho})
$; and by Theorem~\ref{th:mmt}, where each factor yields the sign $(-1)^{\frac{d}{2}\binom{n_\rho}{2}}$, we conclude
$$
    \dett_d(C^G(x)) 
        = (-1)^{\frac{d}{2}(|G|- t)/2} 
        \left(|G|^{|G|}
    \prod_{\rho} 
        \frac{
            h_{n_\rho\times n_\rho}
            }{n_\rho^{n_{\rho}}}
            \right)^{\frac{d}{2}-1}
            \prod_{\rho \in \hat{G}} \det(\rho(x))^{n_
    \rho},
$$
where $\rho(x) = \sum_g x_g \rho(g)$. The constant term  $H_G$ is clearly non-vanishing.
\end{proof}

\section{Algebra tensor}\label{sec:algebra}
Let $A$ be an associative (not necessarily semisimple) algebra with a given basis $\{a_1,\ldots,a_n\}$, let $\{x_1,\ldots,x_n\}$ be the set of variables associated with this basis, and let $\mathrm{Aut}(A)$ be the algebra automorphism group. Recall that the tensor $C^{A,a}(x) = \ell_x(a_{i_1}\cdots a_{i_d})$, where $\ell_x = \sum_{i=1}^n x_i a^*_i$.

\subsection{Invariance}
Let us show that the vanishing property of $C^{A,a}(x)$ does not depend on the basis. Let $a'_1,\ldots,a'_n$ be a different basis, with variables $x'_1,\ldots,x'_n$ associated with it.

\begin{lemma}\label{lemma:basis}
    As polynomials, we have $\det(C^{A,a}(x)) \neq 0 \iff \det(C^{A',a'}(x')) \neq 0$. 
\end{lemma}
\begin{proof}
    Let $g \in \mathrm{GL}(n,\mathbb{C})$ be the change of basis matrix, i.e. $a'_i = g \cdot a_i$. Then for the dual bases we have $(a')^*_i = a^* g^{-1}$. Therefore, the associated linear forms $\ell_x$ and $\ell'_x$ are related as follows:
    $
        \ell'_{x'} 
        := \sum_{k=1}^n (a')^*_k x'_k
        = \sum_{k=1}^n (a^* g^{-1})_k x'_k = 
        \sum_{i=1}^n a^*_i (g^{-1} x')_i = \ell_{g^{-1}x'}
    $, thus $x = g^{-1}x'$. Moreover,
    $$
        C^{A,a'}(x')_{i_1,\ldots,i_d} = \ell'_{x'}(a'_{i_1}\cdots a'_{i_d}) = \ell_{g^{-1} x}(g a_{i_1} \cdots g a_{i_d}) = (g^{\otimes d} \cdot C^{A,a}(g^{-1} x'))_{i_1,\ldots,i_d}.
    $$
    Thus $\dett_d(C^{A,a'}(x')) = \dett_2(g)^d \dett_d(C^{A,a}(g^{-1}x))$, where $\det_2(g) \neq 0$.
\end{proof}
This justifies the notation $C^{A}(x)$ without specifying the basis.

\subsection{Semisimple algebras}
Let $A$ be semisimple.
By Wedderburn theorem, there is an algebra isomorphism $U$ ($=\Phi$ from the introduction) which maps $A$ to the direct sum of central simple algebras
$$
    U: A \to B:=\bigoplus_{\rho \in \widehat{A}} B_\rho, \qquad B_\rho = \mathrm{Mat}({n_\rho}, \mathbb{C}).
$$
The algebra $B$ has the natural matrix units basis $\{E^{\rho}_{i,j}\}$ indexed by triplets $(\rho,i,j)$, where $\rho \in \widehat{A}$ and $i,j \in [n_\rho]$. 
Such $U$ is not unique. 
It maps the $d$-fold iterated multiplication tensor $\mu_A: A^{\otimes d} \otimes A^*$ to the multiplication tensor of $B$ by coordinatewise action:
$$
    U: \mu_A \mapsto \mu_B, 
    \qquad (U^{\times d}, U^{-T}) \cdot \mu_A = \mu_B.
$$
\begin{theorem}
    We have,
    $$
        \dett_d(C^A(x)) 
        = H_{A} \cdot \prod_{\rho \in \widehat{A}} \left(\ \dett_2(y_{ij}^{\rho})\ \right)^{n_\rho},
    $$
    where $y = Ux$.
\end{theorem}

\begin{proof}
Recall, $C^A(x) = \sum_{k=1}^n (\mu_A)^k x_k$. Then after the transformation $U$ in each of $d$ modes
\begin{align*}
    (U^{\times d} \cdot C^A(x))_{i_1,\ldots,i_d} = \sum_{j_1,\ldots,j_d} \prod_{\ell=1}^d U_{i_\ell j_\ell} \sum_{k=1}^n (\mu_A)^{k}_{j_1,\ldots,j_d} x_k = \left(\sum_{k=1}^n x_k ((U^{\times d}, \mathrm{I}) \cdot \mu_A)^k\right)_{i_1,\ldots,i_d},
\end{align*}
where $(U^{\times d}, \mathrm{I})$ acts on $\mu_A \in A^{\otimes d} \otimes A^*$ by leaving the output coordinate unchanged. Completing the action to $(U^{\times d}, U^{-1})$ by replacing the $x$-coordinates with the new coordinates $y = (y^{\rho}_{ij})_{\rho,i,j} = Ux$, we obtain
\begin{align*}
    U^{\times d} \cdot C^A(x) 
    = \sum_{k=1}^n x_k ((\mathrm{I}^{\times d}, U^{T}) \cdot (U^{\times d}, U^{-T}) \cdot \mu_A)^k
    = \sum_{i=1}^n (\sum_{k=1}^n x_k U^{T}_{ki}) (\mu_B)^i = C^B(y).
\end{align*}
Then
$
    C^B(y) = \bigoplus_{\rho \in \widehat{A}} C^{B_\rho}(y^\rho_{ij}) 
$
and in the basis form:
$$
    C^{B_\rho}(y) = \sum_{i_1,i_2,\ldots,i_d,j_d} E^{\rho}_{i_1i_2} \otimes E^{\rho}_{i_2 i_3} \ldots \otimes E^\rho_{i_{d} j_{d}} y^{\rho}_{i_1,j_{d} } \in (\mathbb{C}^{n_\rho^2})^{\otimes d}.
$$
Hence, 
$$
    \dett_d(C^A(x)) = \det(U)^{d} \cdot \dett_d(C^B(y)) = \det(U)^d \cdot \prod_{\rho \in \widehat{A}} \dett_d(C^{B_\rho}(y^{\rho})).
$$
Then by Lemma~\ref{lemma:trace} and Theorem~\ref{th:mmt} we have $\dett_d(C^{\mathrm{Mat}_{n_\rho}}(y^{\rho})) = \dett_2(y^{\rho})^{n_\rho} (h_{n_\rho \times n_\rho})^{d/2-1}$. 

Moreover, by Skolem-Noether theorem, if we choose another basis of the matrix units $\{E'_{ij}\}$ in some fixed $\rho$-block, then there exists $S \in \mathrm{GL}(V_\rho)$, s.t. $E'_{ij} = S E_{ij} S^{-1} = \sum_{a,b} S_{ai} E_{ab} S^{-1}_{jb}$, where $E_{ij}:=E^\rho_{ij}$. 
Then, $y'_{ij} = \ell_y(E'_{ij})=\sum_{a,b} S_{ai} y_{ab} S^{-1}_{jb} = S y_{ij} S^{-1}$, therefore determinant changes as follows $\dett_2(y'_{ij}) = \dett_2(S^T (y_{ij}) S^{-T}) = \dett_2(y_{ij})$.
\end{proof}

\subsection{Nonsemisimple algebras}
The Jacobson radical $J(A)$ of an associative algebra $A$ is the largest two-sided nilpotent ideal. The algebra is semisimple if and only if $J(A) = 0$.
We will show that the algebra hyperdeterminant vanishes identically when $J(A) \neq 0$ using methods from geometric invariant theory. The proof is similar to that in \cite{lys} for multiplication tensor $\mu_A$.

\begin{theorem}
    Let $d > 2$ be even. Then, the Jacobson radical is nontrivial $J(A) \neq 0$ if and only if as a polynomial $\dett_d(C^A(x)) \equiv 0$.
\end{theorem}
\begin{proof}
    If Jacobson radical is trivial, then the algebra is semisimple, thus by Theorem~\ref{th:alg}, polynomial $\det_d(C^A(x))$ does not vanish. Now, assume $J(A) \neq 0$. 
    By Lemma~\ref{lemma:basis}, to show the vanishing it is enough to display it on a specific basis. 
    We will choose a basis that respects the so-called radical filtration.
    
    Let $J:=J(A)$ with $J^{r} = 0$ for minimal $r \ge 2$. Consider the filtration 
    $$
        A = J^0 \supset J^1 \supset J^2 \supset \cdots \supset J^r = 0,
    $$ 
    where $J^k = J \cdots J$ ($k$ times).
    By the Wedderburn--Malcev theorem, the quotient $S_0 = A / J(A)$ is a semisimple subalgebra. 
    Let $S_k$ be any complement of $J^{k+1}$ in $J^k$, i.e. $J^k = S_k \oplus J^{k+1}$.
    Set $j_k = \dim J^k$ and $s_k = \dim S_k$. 
    Let $\{a_1, \ldots, a_n\}$ be a basis respecting this filtration, i.e. $S_k = \mathrm{span}\{a_i \mid j_{k} \ge i > j_{k+1}\}$. For each $i$, let $a_i \in S_{\nu(i)}$, where $\nu(i) \ge 0$ indicates the component to which it belongs, namely its filtration degree.

    Consider the tensor $C^A(x) = \sum_k \mu^k x_k$ with respect to the chosen basis, where $\{\mu^k_{i_1,\ldots,i_d}\}$ are the coordinates of the multiplication tensor. Note that
    \begin{align}\label{eq:vanish}
        \text{if }\quad \mu^k_{i_1,\ldots,i_d} \neq 0, \qquad\text{then}\quad \nu(k) \ge \nu(i_1) + \ldots + \nu(i_d).
    \end{align}
    Indeed, by definition, multiplication satisfies $J^k \cdot J^{\ell} \subseteq J^{k+\ell}$; therefore, any nonzero coefficient in the product $a_{i_1} \cdots a_{i_d}$ has the indicated lower bound on its filtration degree.
    Let $D := \sum_{i=1}^n \nu(i) = \sum_{k=0}^{r-1} k s_k$ be the sum of the degrees. Since $J \neq 0$, we have $D > 0$. Consider the transformation $g(t) \in \mathrm{SL}(n, \mathbb{C})$ defined by rescaling each basis vector:
    $$
        g(t): a_i \mapsto t^{w_i} a_i,\quad \text{where }w_i = n\cdot \nu(i) - D.
    $$ 
    Indeed, $\dett(g(t)) = \prod_{i=1}^n t^{\sum (n \cdot \nu(i) - D)} = t^{nD - nD} = 1$.
    Define $C^A(x;t) := g(t)^{\otimes d} \cdot C^A(x)$. Then, by analogous computations to those in Lemma~\ref{lemma:basis},
    $$
        C^A(x;t)_{i_1,\ldots,i_d} = \sum_k \mu^k_{i_1,\ldots,i_d} x_k \cdot t^{w_k - w_{i_1} - \ldots - w_{i_d}},
    $$
    where the exponent of $t$ is $w_k - \sum_{\ell=1}^d w_{i_\ell} = n \cdot (\nu(k) - \sum_{\ell = 1}^d \nu(i_\ell)) + (d-1)D > 0$ by \eqref{eq:vanish} and $D > 0$. Then the limit $t \to 0$ exists and yields $\lim_{t\to 0} C^A(x,t) = 0$. This implies that the closure $ \overline{\mathrm{SL}(n,\mathbb{C})^{\times d}\ C^A(x)}$, contains $0$, and hence $C^A(x)$ vanishes on continious function $\dett_d$.
\end{proof}

\begin{remark}
    In fact this show that the tensor $C^A(x)$ is unstable -- vanishes on any hyperdeterminant, i.e. it lies in the null cone, the zero set of all hyperdeterminants. The argument above is similar to the proof of null-cone membership for the multiplication tensor $\mu_A$ in \cite{lys}.
\end{remark}

\section*{Acknowledgements}
The author is grateful to Maciej Do\l{}ęga, Damir Yeliussizov, Joachim Jelisiejew, Askar Dzhumadildayev
 and Pietro Campochiaro for useful conversations and comments. This research was funded in part by {\it Narodowe Centrum Nauki}, grant 2021/42/E/ST1/00162.


\begin{thebibliography}{99}
\bibitem{ay}
Amanov, A., Yeliussizov, D. (2023).
Tensor slice rank and Cayley's first hyperdeterminant.
\textit{Linear Algebra and its Applications}, \textbf{656}, 224--246.

\bibitem{ay2}
Amanov, A., Yeliussizov, D. (2023).
Fundamental invariants of tensors, Latin hypercubes, and rectangular Kronecker coefficients.
\textit{International Mathematics Research Notices}, \textbf{2023}(20), 17552--17599.

\bibitem{ay3}
Amanov, A., Yeliussizov, D. (2025).
Some unimodal sequences of Kronecker coefficients.
\textit{Mathematische Zeitschrift}, \textbf{309}(1), 8.

\bibitem{ay4}
Amanov, A., Yeliussizov, D. (2025).
Highest weight vectors of tensors.
\textit{arXiv preprint} arXiv:2504.15413.

\bibitem{barvinok}
Barvinok, A. I. (1995).
New algorithms for linear $k$-matroid intersection and matroid $k$-parity problems.
\textit{Mathematical Programming}, \textbf{69}(1), 449--470.

\bibitem{blas}
Bl{\"a}ser, M. (2013).
Fast matrix multiplication.
\textit{Theory of Computing}, \textbf{5}, 1--60.
(Graduate Surveys)

\bibitem{lys}
Bl{\"a}ser, M., Lysikov, V. (2020).
Slice rank of block tensors and irreversibility of structure tensors of algebras.
In \textit{45th International Symposium on Mathematical Foundations of Computer Science (MFCS 2020)} (pp.~17:1--17:15).
Schloss Dagstuhl--Leibniz-Zentrum f{\"u}r Informatik.

\bibitem{bi}
B{\"u}rgisser, P., Ikenmeyer, C. (2017).
Fundamental invariants of orbit closures.
\textit{Journal of Algebra}, \textbf{477}, 390--434.

\bibitem{cay}
Cayley, A. (1843).
On the theory of determinants.
\textit{Transactions of the Cambridge Philosophical Society}, \textbf{8}, 1--16.

\bibitem{cay2}
Cayley, A. (1845).
On the theory of linear transformations.
\textit{Cambridge Mathematical Journal}, \textbf{4}, 193--209.

\bibitem{cohn}
Cohn, H., Umans, C. (2003, October).
A group-theoretic approach to fast matrix multiplication.
In \textit{Proceedings of the 44th Annual IEEE Symposium on Foundations of Computer Science} (pp.~438--449).
IEEE.

\bibitem{conrad}
Conrad, K. (1998).
On the origin of representation theory.
\textit{L'Enseignement Math\'ematique}, \textbf{44}, 361--392.

\bibitem{etingof}
Etingof, P. I., Golberg, O., Hensel, S., Liu, T., Schwendner, A., Vaintrob, D., Yudovina, E. (2011).
\textit{Introduction to representation theory}.
(Vol.~59).
American Mathematical Society.

\bibitem{frob}
Frobenius, G. (1896).
{\"U}ber Gruppencharaktere.
\textit{Sitzungsberichte der K{\"o}niglich Preussischen Akademie der Wissenschaften zu Berlin}.

\bibitem{formsibl}
Formanek, E., Sibley, D. (1991).
The group determinant determines the group.
\textit{Proceedings of the American Mathematical Society}, \textbf{112}(3), 649--656.

\bibitem{fulton}
Fulton, W. (1997).
\textit{Young tableaux: With applications to representation theory and geometry}.
(London Mathematical Society Student Texts, Vol.~35).
Cambridge University Press.

\bibitem{gkz}
Gelfand, I. M., Kapranov, M. M., Zelevinsky, A. V. (1994).
Hyperdeterminants.
In \textit{Discriminants, Resultants, and Multidimensional Determinants} (pp.~444--479).
Birkh{\"a}user Boston.

\bibitem{rota}
Huang, R., Rota, G. C. (1994).
On the relations of various conjectures on Latin squares and straightening coefficients.
\textit{Discrete Mathematics}, \textbf{128}(1--3), 225--236.

\bibitem{kum}
Kumar, S., Landsberg, J. M. (2015). Connections between conjectures of Alon–Tarsi, Hadamard–Howe, and integrals over the special unitary group. \textit{Discrete Mathematics}, 338(7), 1232-1238.

\bibitem{land}
Landsberg, J. M. (2017).
\textit{Geometry and complexity theory}.
(Cambridge Studies in Advanced Mathematics, Vol.~169).
Cambridge University Press.


\bibitem{lzx}
Li, X., Zhang, L., Xia, H. (2025).
Two classes of minimal generic fundamental invariants for tensors.
\textit{Linear Algebra and its Applications}, \textbf{720}, 174--212.

\bibitem{luq-thi}
Luque, J. G., Thibon, J. Y. (2003).
Hankel hyperdeterminants and Selberg integrals.
\textit{Journal of Physics A: Mathematical and General}, \textbf{36}(19), 5267.

\bibitem{matsumoto}
Matsumoto, S. (2008).
\textit{Hyperdeterminantal expressions for Jack functions of rectangular shapes}.
\textit{Journal of Algebra}, \textbf{320}(2), 612--632.

\bibitem{stein}
Steinberg, B. (2022).
Factoring the Dedekind--Frobenius determinant of a semigroup.
\textit{Journal of Algebra}, \textbf{605}, 1--36.

\bibitem{yel}
Yeliussizov, D. (2023).
Stability of the Levi-Civita tensors and an Alon--Tarsi type theorem.
\textit{Comptes Rendus. Math\'ematique}, \textbf{361}(G8), 1367--1373.

\bibitem{zappa}
Zappa, P. (1997).
The Cayley determinant of the determinant tensor and the Alon--Tarsi conjecture.
\textit{Advances in Applied Mathematics}, \textbf{19}(1), 31--44.
\end{thebibliography}
\end{document}